\documentclass[12pt]{amsart}
\advance\hoffset by -0.5 truecm
\usepackage{amssymb}
\usepackage{xcolor}

\usepackage[colorlinks=true]{hyperref}
\hypersetup{urlcolor=blue, citecolor=red}

\newtheorem{Lemma}{Lemma}[section]
\newtheorem{Theorem}[Lemma]{Theorem}
\newtheorem{Proposition}[Lemma]{Proposition}
\newtheorem{Corollary}[Lemma]{Corollary}
\newtheorem{Remark}[Lemma]{Remark}
\newtheorem{Definition}[Lemma]{Definition}

\def\cD{\mathcal D}

\def\bfG{\mathbf G}

\def\R{\mathbb R}

\def\C{\mathbb C}
\def\K{\mathbb K}
\def\Z{\mathbb Z}

\newcommand{\de}[2]{\frac{\partial #1}{\partial #2}}
\newcommand{\p}{\partial}

\newcommand{\<}{\langle}
\renewcommand{\>}{\rangle}

\renewcommand{\(}{\left(}
\renewcommand{\)}{\right)}
\renewcommand{\div}{\operatorname{div}}

\begin{document}
\title[Tensor tomography for Gaussian thermostats]{Invariant distributions and tensor tomography for Gaussian thermostats}

\author[Yernat M. Assylbekov]{Yernat M. Assylbekov}
\address{Department of Mathematics, University of Washington, Seattle, WA 98195-4350, USA}
\email{y\_assylbekov@yahoo.com}

\author[Hanming Zhou]{Hanming Zhou}
\address{Department of Mathematics, University of Washington, Seattle, WA 98195-4350, USA\\
Current address: Department of Pure Mathematics and Mathematical Statistics, University of Cambridge, Cambridge CB3 0WB, UK}
\email{hz318@dpmms.cam.ac.uk}


\begin{abstract}
In this paper we consider the Gaussian thermostat ray transforms on both closed Riemannian surfaces and compact Riemannian surfaces with boundary. We establish certain results on the injectivity of the thermostat ray transform and the surjectivity of its adjoint.
\end{abstract}

\maketitle


\section{Introduction and statement of results}
\subsection{Gaussian Thermostats}
Let $(M,g)$ be a compact oriented Riemannian manifold (with or without boundary) and $E$ be a smooth vector field on $M$ (called the {\it external field}). A parameterized curve $\gamma(t)$ on $M$ satisfying the equation
\begin{equation}\label{Gaussian thermostat}
D_t\dot \gamma=E(\gamma)-\frac{\<E(\gamma),\dot\gamma\>}{|\dot\gamma|^2}\dot\gamma.
\end{equation}
is called a {\it thermostat geodesic}. Here and in what follows $D_t$ denotes the covariant derivative along $\gamma$. This differential equation defines a flow $\phi_t=(\gamma(t),\dot{\gamma}(t))$ on $SM$ (the unit sphere bundle of $M$) which is called a {\it Gaussian thermostat} (or {\it isokinetic dynamics}, see \cite{Ho}). The flow $\phi$ reduces to the geodesic flow when $E=0$. As in the case of geodesic flows, Gaussian thermostats are reversible in the sense that the flip $(x, v)\mapsto(x,-v)$ conjugates $\phi_t$ with $\phi_{-t}$. We denote the Gaussian thermostat by $(M,g,E)$ and the generating vector of the thermostat flow by $\bfG_E$, which is a vector field on $SM$.

In this paper we will consider the case when $M$ is a surface (i.e. 2-dimensional manifolds). Then for $(x,v)\in SM$ we can write
$$
E(x)=\<E(x),v\>v+\<E(x),iv\>iv,
$$
where $i$ indicates the rotation by $\pi/2$ according to the orientation of $M$. Thus on surfaces, the equation \eqref{Gaussian thermostat} can be rewritten as
\begin{equation}\label{thermo}
D_t\dot\gamma=\lambda(\gamma,\dot\gamma)i\dot\gamma,
\end{equation}
where
\begin{equation}\label{lambda}
\lambda(x,v):=\<E(x),iv\>.
\end{equation}
Notice that for Gaussian thermostats, $\lambda$ corresponds to a 1-form on $M$. If $\lambda$ is a smooth function on $M$, \eqref{thermo} defines the magnetic flow on surfaces associated with the magnetic field $\Omega=\lambda\,d{\rm Vol}_g$, where $d{\rm Vol}_g$ is the area form of $M$. One can consider a general function $\lambda\in C^{\infty}(SM)$, we call the induced flow a {\it generalized thermostat}. 

In dynamical systems, Gaussian thermostats provide interesting models in non-equilibrium statistical mechanics \cite{G,GR,R}. Gaussian thermostats also arise in geometry as the flows of metric connections with non-zero torsion; see \cite{W}.



\subsection{Thermostat ray transforms}
Given a Gaussian thermostat $(M,g,E)$, we define the {\it thermostat ray transform} of a smooth function $\varphi$ on $SM$ to be
$$
I\varphi(\gamma):=\int_0^T \varphi(\gamma(t),\dot\gamma(t))\,dt.
$$
When $M$ is closed, $\gamma$ is a closed thermostat geodesic with period $T$. A basic question of integral geometry is whether the ray transform is injective. Of course, this question makes sense only in the case when the flow has sufficiently many closed orbits. Anosov flows constitute wide class of flows with sufficiently many closed orbits. Recall that a Gaussian thermostat $(M,g,E)$ is said to be {\it Anosov} if there is a continuous invariant splitting $T(SM)=\mathbb{R} \bfG_E\oplus E^{u}\oplus E^{s}$ in such a way that there are constants $C>0$ and $0<\rho<1<\eta$ such that 
for all $t>0$ we have
\[\|d\phi_{-t}|_{E^{u}}\|\leq C\,\eta^{-t}\;\;\;\;\mbox{\rm
and}\;\;\;\|d\phi_{t}|_{E^{s}}\|\leq C\,\rho^{t},\]
where norms are taken with respect to the Sasaki type Riemannian metric on $SM$.

There is a natural obstruction to the injectivity of the ray transform, i.e. the functions of the type $\varphi=\bfG_E u$ with $u\in C^\infty(SM)$. However, in applications one often needs to invert the ray transform of functions on $SM$ arising from symmetric tensor fields. Therefore, we consider this particular case which is known as the tensor tomography problem.

Let $\varphi=\varphi_{i_1\dots i_m}\,dx^{i_1}\otimes \dots\otimes dx^{i_m}$ be a smooth symmetric $m$-tensor field on $M$. $\varphi$ induces a smooth function $\hat\varphi\in C^\infty(SM)$ defined by
$$
\hat\varphi(x,v):=\varphi_{i_1\dots i_m}(x)\, v^{i_1} \cdots v^{i_m},\quad (x,v)\in SM.
$$
In what follows we will drop the hat, and we hope that it will be clear from the context when we mean the function on $SM$ induced by the tensor. By $C^\infty(S_m(M))$ we denote the bundle of smooth symmetric $m$-tensor fields on $M$.

We say that $I_m$ ($I$ acting on $m$-tensors) is {\it $s$-injective} if $I_m\varphi\equiv 0$ implies that $\varphi=\bfG_E h$ for some $h\in C^\infty(S_{m-1}(M))$. The tensor tomography problem asks under what conditions $I_m$ is $s$-injective. The tensor tomography problem on Anosov surfaces was studied in \cite{DS, SU, PSU, GK, CS, Gu}, and \cite{DP2, Ain} for magnetic Anosov surfaces. In this paper, we will focus on the tensor tomography problem for Gaussian thermostats. In \cite{DP} Dairbekov and Paternain proved the $s$-injectivity of $I_m$ for $m=0,1$, but considering more general Anosov thermostats. In \cite{AD} Assylbekov and Dairbekov extended this result to the case when the Riemannian metric is replaced by a Finsler metric. They showed that for $m=0$ injectivity results hold even when the flow is not Anosov, but simply has no conjugate points. When $m=2$, Jane and Paternain \cite{JP} proved $s$-injectivity under the assumption that the external field is divergence free and the surface has negative Gaussian curvature.

Similarly there is a tensor tomography problem for Gaussian thermostats on compact Riemannian surfaces with boundary. In this case, the ray transform is along thermostat geodesics joining boundary points. For the boundary case, the tensor tomography problem for geodesic flows has been extensively studied, see e.g. \cite{Mu, AR, PU2, Sh2, SU1, SU2, PSU2, MSU, PZ} and the references therein. The case of magnetic flows was considered in \cite{DPSU, Ain2}. We will study the boundary case in the last section of the paper.

\subsection{Injectivity results for $I_m$}
For the case of Gaussian thermostats we obtain several injectivity results of the thermostat ray transform under various assumptions. In order to state these results we need to introduce some notations.

Since $M$ is assumed to be oriented there is a circle action on the fibres of $SM$ with infinitesimal generator $V$ called the vertical vector field. Let $X$ denote the generator of the geodesic flow of $g$. We complete $X,V$ to a global frame of $T(SM)$ by defining the vector field $X_\perp:=[V,X]$, where $[\cdot,\cdot]$ is the Lie bracket for vector fields. In this global frame, the generating vector field $\bfG_E$ for a Gaussian thermostat $(M,g,E)$ equals $X+\lambda V$. 

Define the {\it thermostat curvature} to be the quantity $\mathbb{K}:=K-\div_g E$, where $K$ is the Gaussian curvature of the surface $(M,g)$. The quantity $\mathbb{K}$ can also be written as $K+X_\perp\lambda+\lambda^2+\bfG_E V\lambda$. Notice that $\mathbb{K}$ is a smooth function on $M$. Following \cite{PSU}, we introduce a definition involving a modified thermostat Jacobi equation.
\begin{Definition}{\rm
Let $(M,g,E)$ be a Gaussian thermostat on a closed oriented Riemannian surface. We say that $(M,g,E)$ {\it has no $\beta$-conjugate points} if for any thermostat geodesic $\gamma$, all non-trivial solutions to the $\beta$-Jacobi equation along $\gamma$
\begin{equation}\label{b-Jeq}
\ddot y-V(\lambda)\dot y+(\beta \mathbb{K}-\bfG_E V(\lambda))y=0
\end{equation}
vanish at most once. The {\it terminator value} of $(M,g,E)$ is defined to be
$$
\beta_{\rm ter}=\sup\{\beta\in [0,\infty]:(M,g,E)\text{ has no $\beta$-conjugate points}\}.
$$}
\end{Definition}

It is clear that $1$-conjugate points are the same as usual conjugate points for thermostat geodesics (see \cite{AD1,P} for more details on the thermostat Jacobi equation).

\begin{Theorem}\label{th1}
Let $(M,g,E)$ be an Anosov Gaussian thermostat on a closed oriented Riemannian surface. Assume that $\beta_{\rm ter}\geq (m+1)/2$ for some integer $m\geq 2$, then $I_m$ is $s$-injective.
\end{Theorem}

Theorem \ref{th1} generalizes the corresponding injectivity result in \cite{PSU} which is for the geodesic ray transform. In particular, \cite{PSU} showed the $s$-injectivity of $I_2$ on Anosov surfaces, before which it was only known for Anosov surfaces without focal points \cite{SU}. Recently Guillarmou \cite{Gu} settled the tensor tomography problem on Anosov surfaces for tensor fields of any order. It was proved that $s$-injectivity of $I_2$ also holds on 2D Anosov magnetic surfaces \cite{Ain}. The problem of proving $s$-injectivity of $I_2$ for 2D Anosov Gaussian thermostats without the assumption on terminator values is still open. The difficulty comes from the fact that in general $V(\lambda)$ is nonzero for Gaussian thermostats, see Section 2 for details. 

The condition on $\beta_{\rm ter}$ is closely related to the works \cite{D,Pe} where absence of $\beta$-conjugate points also appears in the case of geodesic flows on manifolds with boundary. When the thermostat curvature is non-positive, i.e. $\mathbb{K}\leq 0$, it is not difficult to see that $\beta_{\rm ter}=\infty$. We get the following result as a corollary of Theorem \ref{th1}, and it generalizs an earlier result \cite{JP} which is for $m=2$.

\begin{Corollary}\label{cor1}
Let $(M,g,E)$ be an Anosov Gaussian thermostat on a closed oriented Riemannian surface of non-positive thermostat curvature. Then $I_m$ is $s$-injective for any integer $m\ge 2$.
\end{Corollary}

According to the result of Wojtkowski \cite[Theorem~5.2]{W} a Gaussian thermostat on a closed surface with negative thermostat curvature is always Anosov.

\begin{Corollary}\label{cor2}
Let $(M,g,E)$ be a Gaussian thermostat on a closed oriented Riemannian surface of negative thermostat curvature. Then $I_m$ is $s$-injective for any integer $m\ge 2$.
\end{Corollary}

At the end of the paper, we apply the ideas from Anosov Gaussian thermostats to study the injectivity of the thermostat ray transform on compact surfaces $(M,g)$ with smooth boundaries. We will focus on a class of Gaussian thermostats which are called {\it simple} Gaussian thermostats (see Section 8 for precise definition). Roughly speaking, simple Gaussian thermostats are the analogues of Anosov Gaussian thermostats for manifolds with boundary. 

Simplicity is related to the boundary rigidity problem \cite{Mi} which is a motivation for the tensor tomography problem. It was shown by Pestov and Uhlmann \cite{PU1} that simple surfaces are boundary rigid. Later this rigidity result was generalized to 2D simple magnetic systems \cite{DPSU} and 2D simple systems involving magnetic fields and potentials \cite{AZ}.

\begin{Theorem}\label{th2}
Let $(M,g,E)$ be a simple Gaussian thermostat on a compact oriented Riemannian surface with boundary. Assume that $\beta_{\rm ter}\geq (m+1)/2$ for some integer $m\geq 2$, then $I_m$ is $s$-injective.
\end{Theorem}

In particular, $\beta_{\rm ter}=\infty$ when the thermostat curvature is non-positive.

\begin{Corollary}
Let $(M,g,E)$ be a simple Gaussian thermostat on a compact oriented Riemannian surface with boundary of non-positive thermostat curvature. Then $I_m$ is $s$-injective for any integer $m\geq 2$.
\end{Corollary}

The tensor tomography problems for simple surfaces \cite{PSU2} and 2D simple magnetic systems \cite{Ain2} were proved without curvature assumptions, using a different method which was developed for the boundary case. It is an interesting problem to show $s$-injectivity of $I_m, m\geq 2$ for simple Gaussian thermostats on surfaces.

For manifolds with boundaries, there are also local tensor tomography problems, i.e. whether one can determine a symmetric tensor near a boundary point, up to the natural obstruction, from its integrals along curves near this point? For manifolds of dimension three and higher, there are recent works by Uhlmann and Vasy \cite{UV}, Stefanov, Uhlmann and Vasy \cite{SUV} for the geodesic case, and Zhou \cite[Appendix]{UV} for general smooth curves, including the thermostats. However, the local problem for surfaces is still open.



\subsection{Invariant distributions}
One key ingredient in the proof of the $s$-injectivity of $I_2$ for the case of Anosov surfaces by Paternain, Salo and Uhlmann \cite{PSU} was the surjectivity of the adjoint of the geodesic ray transform acting on $1$-forms. The problem of the surjectivity of the adjoint is interesting in its own right. We also investigate the surjectivity of the adjoint of the thermostat ray transform. However, in the case of thermostat flows the surjectivity of $I_1^*$ seems not enough for proving the $s$-injectivity of $I_2$. In general, thermostats do not preserve the Liouville measure on $SM$ unless $E\equiv0$ (see \cite{DP1,DP}). This is a crucial difference from the case of geodesic flows and magnetic flows, and this makes the problem much harder.

Since the thermostat ray transforms $I_0$ and $I_1$ are $s$-injective for two-dimensional Anosov thermostats, one can consider the surjectivity of $I^*_0$ and $I^*_1$. One of the aims of the current paper is to show that $I^*_0$ and $I^*_1$ are indeed surjective. To study the adjoints, we pause to briefly introduce distributions on $SM$.

Let $\gamma$ be a closed thermostat geodesic and $\delta_\gamma$ denote the measure on $SM$ which corresponds to integrating over $(\gamma,\dot\gamma)$ on $SM$. We can define the thermostat ray transform by the distributional pairing
$$
I\varphi(\gamma)=\<\delta_\gamma,\varphi\>.
$$
Denote by $\cD'(SM)$ the space of distributions on $C^\infty(SM)$. Both of these spaces are reflexive, so the dual of $\cD'(SM)$ is $C^\infty(SM)$. Any differential operator $P$ can act on a distribution $\mu\in\cD'(SM)$ via duality, that is $\<P\mu,\varphi\>:=\<\mu,P^*\varphi\>$ for any $\varphi\in C^\infty(SM)$. Since $\bfG_E=-(\bfG_E+V(\lambda))^*$ (see Section~2), we define the following subspace of $\cD'(SM)$:
$$
\cD'_{\rm inv}(SM):=\{\mu\in\cD'(SM):(\bfG_E+V(\lambda))\mu=0\}.
$$
Hence a distribution $\mu$ is in $\cD'_{\rm inv}(SM)$ if and only if $\<\mu,\bfG_E\varphi\>=0$ for all $\varphi\in C^\infty(SM)$. This agrees with the definition of the thermostat ray transform given by the distributional pairing.

Without loss of generality we can consider the thermostat ray transform $I$ as the map
$$
I:C^\infty(SM)\to L(\cD'_{\rm inv}(SM),\R),\quad I\varphi(\mu)=\<\mu,\varphi\>\text{ for }\mu\in\cD'_{\rm inv}(SM).
$$
By $L(F,\R)$ we mean the space of continuous linear maps from a locally convex topological space $F$ to $\R$. Equip this space with the weak* topology, then $I$ becomes a continuous linear map from a Frech\'e{}t space into $L(\cD'_{\rm inv}(SM),\R)$ which is locally convex. Since $\cD'_{\rm inv}(SM)$ is a closed subspace of a reflexive space $\cD'(SM)$, it is also reflexive. Therefore, the dual of $L(\cD'_{\rm inv}(SM),\R)$ is the space of invariant distributions $\cD'_{\rm inv}(SM)$. This implies that the adjoint of the thermostat ray transform $I$ is the map
$$
I^*:\cD'_{\rm inv}(SM)\to \cD'(SM),\quad \<I^*\mu,\varphi\>=\<\mu,I\varphi\>\text{ for }\varphi\in C^\infty(SM).
$$
On an oriented surface any $u\in C^\infty(SM)$ admits a Fourier expansion $u=\sum_{m\in \Z}u_m$ (see Section~2) where
$$
u_m(x,v):=\frac{1}{2\pi}\int_0^{2\pi}u(\rho_t(x,v))e^{-imt}\,dt,
$$
and $\rho_t$ is the flow generated by $V$. One can use duality to decompose a distribution into its Fourier components. That is, if $\mu\in\cD'(SM)$ then $\<\mu_k,\varphi\>=\<\mu,\varphi_k\>$ for all $\varphi\in C^\infty(SM)$. Now we can give the statements of our results which express the surjectivities of $I_0^*$ and $I_1^*$ in terms of the existence of some invariant distributions.

\begin{Theorem}\label{th4}
Let $(M,g,E)$ be an Anosov Gaussian thermostat on a closed oriented Riemannian surface. Given $f\in C^\infty(M)$, there exists $w\in H^{-1}(SM)$ with $(\bfG_E+V(\lambda))w = 0$ and $w_0 = f$.
\end{Theorem}

As was explained in \cite{PSU}, by the ergodicity of Anosov flows, the only $L^2$ solutions to $Xw=0$ on geodesic flows are constants. Therefore, the optimal regularity that we can expect for solutions to $(\bfG_E+V(\lambda))w=0$ is $H^{-1}$.

\begin{Theorem}\label{th5}
Let $(M,g,E)$ be an Anosov Gaussian thermostat on a closed oriented Riemannian surface. For a given solenoidal $1$-form $\alpha$ (i.e. divergence free), there exists $w\in H^{-1}(SM)$ with $(\bfG_E+V(\lambda))w = 0$ and $w_{-1}+w_1=\alpha$.
\end{Theorem}

One can consider the surjectivity of $I_m^*$ for $m\geq 2$, however the constraint on $m$-tensors may not have explicit geometric meanings as that in the geodesic case. One can also derive surjectivity results on surfaces with boundaries by similar techniques. For the boundary case one should expect to show the existence of smooth invariant functions. This is known for $I_0^*$ and $I^*_1$ on simple manifolds of any dimension, see \cite{PU1} and \cite{DU}. For $I_m^*,\, m\geq 2$, there are results on simple surfaces \cite{PSU4}.

Finally, it's also worth pointing out that recently Paternain, Salo and Uhlmann generalized the techniques for the study of $I$ and $I^*$ on Anosov surfaces to higher dimensional Anosov and simple manifolds \cite{PSU3}.

\section{Pestov identity}

Note that we have a global frame $\{X, X_{\perp},V\}$ for $T(SM)$, which satisfies the structure equations given by $X=[V, X_{\perp}]$, $X_{\perp}=[X,V]$ and $[X, X_{\perp}]=-KV$ where $K$ is the Gaussian curvature of the surface. Using this frame we can define a Riemannian metric on $SM$ by declaring $\{X, X_{\perp}, V\}$ to be an orthonormal basis and the volume form of this metric will be denoted by $d\Sigma^3$.

Recall the generating vector field of a Gaussian thermostat $(M,g,E)$ is $\bfG_E=X+\lambda V$. The fact that $X, X_\perp,V$ are volume preserving implies the following lemma which was proved in \cite[Lemma~3.2]{DP}.
\begin{Lemma}\label{lie}
Let $(M,g,E)$ be a Gaussian thermostat on a closed oriented Riemannian surface. Then the following hold:
$$
L_{{\bfG_E}\,}d\Sigma^3=V(\lambda)\,d\Sigma^3,\quad L_{X_\perp}\,d\Sigma^3=0,\quad L_{V}\,d\Sigma^3=0,
$$
where $L_Z$ denotes the Lie derivative along the vector field $Z$.
\end{Lemma}

For any two functions $u,v:SM\to\C$ define the $L^2$ inner product:
$$
(u,v):=\int_{SM} u\bar v\,d\Sigma^3,
$$
the corresponding norm will be denoted by $\|\cdot\|$.

The space $L^2(SM)$ decomposes orthogonally as a direct sum $$L^2(SM)=\bigoplus_{k\in \mathbb Z}H_k$$ where $H_k$ is the eigenspace of $-iV$ corresponding to the eigenvalue $k$. A function $u\in L^2(SM)$ has a Fourier series expansion $$u=\sum^{\infty}_{k=-\infty} u_k$$ where $u_k\in H_k$, then $\|u\|^2=\Sigma\|u_k\|^2$ with $\|u\|^2=(u, u)^{1/2}$. We denote the subspace $\Omega_k:=H_k\cap C^{\infty}(SM)$.

Consider the isothermal coordinates $(x,y)$ on the surface $(M,g)$ such that the metric can be written as $ds^2=e^{2\rho}(dx^2+dy^2)$ where $\rho\in C^\infty(M,\R)$. This gives coordinates $(x,y,\varphi)$ on $SM$ where $\varphi$ is the angle between a unit vector $v$ and $\de{}{x}$. In these coordinates, the elements in the Fourier expansion of $f=f(x,y,\varphi)$ are given by
$$f_k(x,y,\varphi)=\Big(\frac{1}{2\pi}\int_0^{2\pi}f(x,y,\varphi')e^{−ik\varphi'}\,d\varphi'\Big)e^{ik\varphi}.$$
In particular, for a given symmetric tensor field $f$ of order $m$, $f_k=0$ for $|k|\geq m+1$.

We define the $H^1$-norm of a function $u\in C^\infty(SM)$ as
$$
\|u\|^2_{H^1(SM)}:=\|\bfG_E u\|^2+\|X_\perp u-V(\lambda)Vu\|^2+\|Vu\|^2+\|u\|^2.
$$
Notice that $\|u\|^2_{H^1(SM)}$ is equivalent to the standard $H^1$-norm $\|u\|^2+\|\nabla u\|^2$, where $\nabla u=(Xu,X_{\perp}u,Vu)$.

\begin{Lemma}\label{lemma1}
Let $(M,g,E)$ be a Gaussian thermostat on a closed oriented Riemannian surface. For any two functions $u,v\in C^\infty(SM,\C)$ the following hold
$$
(Vu,v)=-(u,Vv),\quad (X_\perp u,v)=-(u,X_\perp v)$$
{\rm and}$$(\bfG_E u,v)=-(u,\bfG_E v)-(V(\lambda)u,v).
$$
\end{Lemma}
\begin{proof}
We will use the following consequence of Stokes' theorem. Let $N$ be a closed oriented manifold and $\Theta$ be a volume form. Let $\mathfrak{X}$ be a vector field on $N$ and $f\in C^\infty(N)$. Then the following holds
\begin{equation}\label{stokes}
\int_N \mathfrak{X}(f)\Theta=-\int_N f L_{\mathfrak{X}}\Theta.
\end{equation}
Now, the statement of the lemma is the consquence of Lemma~\ref{lie} and \eqref{stokes}.
\end{proof}

In particular, Lemma \ref{lemma1} implies the following expressions for the adjoints
$$
X_\perp^*=-X_\perp,\quad V^*=-V,\quad \bfG_E^*=-(\bfG_E+V(\lambda)).
$$

The following integral identity will play a fundamental role in our arguments. Its proof can be found in \cite[Theorem 3.3]{DP}, which is valid for more general thermostats.
\begin{Theorem}[Pestov identity]\label{pestov}
Let $(M,g,E)$ be a Gaussian thermostat on a closed oriented Riemannian surface. If $u\in C^\infty(SM,\C)$, then
$$
\|\bfG_E Vu\|^2-(\mathbb{K} Vu,Vu)=\|V\bfG_E u\|^2-\|\bfG_E u\|^2.
$$
\end{Theorem}

\begin{Remark}
\rm The Pestov identity above also holds for Gaussian thermostats on a compact oriented surfaces with smooth boundaries provided that $u|_{\p SM}=0$.
\end{Remark}


\section{$\alpha$-controlled thermostats}

For $\alpha\in[0,1]$, we say that a Gaussian thermostat $(M,g,E)$ on a closed surface is {\it $\alpha$-controlled} if for any $u\in C^\infty(SM)$ ( $u\in C^{\infty}_0(SM)$ for compact surfaces with boundaries) the following holds
$$
\|\bfG_E u\|^2-(\K u,u)\ge\alpha\|\bfG_E u\|^2.
$$
It is obvious that if $\mathbb{K}\leq 0$, then $(M,g,E)$ is $1$-controlled.

\begin{Theorem}\label{existence of alpha-control}
Let $(M,g,E)$ be an Anosov Gaussian thermostat on a closed surface. Then there is an $\alpha>0$ such that
$$
\|\mathbf G_E\varphi\|^2-(\K\varphi,\varphi)\ge\alpha\left(\|\mathbf G_E\varphi\|^2+\|\varphi\|^2\right)
$$
for all $\varphi\in C^\infty(SM)$.
\end{Theorem}
\begin{proof}
Consider the following Riccati type equation
$$
\mathbf G_E(r-V(\lambda))+r(r-V(\lambda))+\K=0.
$$
It was shown in \cite{DP} that for Anosov thermostats there are real-valued continuous solutions $r^\pm$ (on $SM$) to this equation, which are differentiable along the thermostat flow and satisfy $r^+-r^->0$. We prove that the following integral identity holds
\begin{equation}\label{non-negative}
\|\mathbf G_E\varphi\|^2-(\K\varphi,\varphi)=\|\mathbf G_E\varphi-r\varphi+V(\lambda)\varphi\|^2,
\end{equation}
where $r=r^\pm$.
\begin{multline*}
|\mathbf G_E\varphi-r\varphi+V(\lambda)\varphi|^2=|\mathbf G_E(\varphi)|^2+|r\varphi|^2+|V(\lambda)\varphi|^2-2\mathfrak{Re}(r\mathbf G_E(\varphi)\,\overline{\varphi})\\
+2\mathfrak{Re}(V(\lambda)\mathbf G_E(\varphi)\,\overline{\varphi})-2rV(\lambda)|\varphi|^2.
\end{multline*}
Since $r$ satisfies the Ricatti equation,
\begin{equation*}
\begin{split}
|\mathbf G_E\varphi-r\varphi+V(\lambda)\varphi|^2=|\mathbf G_E(\varphi)|^2 & -\K|\varphi|^2+|V(\lambda)\varphi|^2\\
& -\mathbf G_E((r-V(\lambda))|\varphi|^2)-rV(\lambda)|\varphi|^2.
\end{split}
\end{equation*}
Integrate this over $SM$ and use \eqref{stokes} together with Lemma~\ref{lie} to derive \eqref{non-negative}.

Let $A:=\mathbf G_E\varphi-r^+\varphi+V(\lambda)\varphi$ and $B:=\mathbf G_E\varphi-r^-\varphi+V(\lambda)\varphi$, the equation \eqref{non-negative} implies $\|A\|=\|B\|$. We obtain the following expressions for $\varphi$ and $\mathbf G_E\varphi$
\begin{align*}
&\varphi=(r^+-r^-)^{-1}(A-B),\\
&\mathbf G_E\varphi=(1-c)A+cB,
\end{align*}
where $c:=\frac{r^+-V\lambda}{r^+-r^-}$. From these equations one concludes that there is an $\alpha>0$ such that
$$
2\alpha\|\varphi\|^2\le \|A\|^2,\quad 2\alpha\|\mathbf G_E\varphi\|^2\le \|A\|^2.
$$
Combining above inequalities with \eqref{non-negative}, this completes the proof.
\end{proof}

\begin{Remark}\label{no b conj impies (b-1)/b-controlled}{\rm
The proof of Theorem~\ref{existence of alpha-control} shows that the following more general statement holds: if there is a bounded measurable function $r:SM\to \R$ such that
$$
\bfG_E (r-V(\lambda))+r(r-V(\lambda))+\beta \mathbb{K}\le 0,
$$
then the Gaussian thermostat $(M,g,E)$ is $(\beta-1)/\beta$-controlled.}
\end{Remark}



\section{Surjectivity of $I^*_0$}

This section is devoted to the surjectivity of the adjoint of the thermostat ray transform acting on functions, i.e. $I^*_0$. To prove the surjectivity of $I^*_0$, we need to study the properties of the operator $P:=V\bfG_E$. Appling Lemma~\ref{lemma1}, it is easy to see that $P^*=(\bfG_E+V(\lambda))V$. If $F$ is a subspace of $\mathcal D'(SM)$, we denote by $F_\diamond$ the subspace of those $v\in F$ such that $\<v,1\>=0$.
\begin{Lemma}\label{P}
Let $(M,g,E)$ be an Anosov Gaussian thermostat on a closed oriented Riemannian surface. Then there is a positive constant $C$ such that
$$
\|u\|_{H^1(SM)}\le C\|Pu\|
$$
for all $u\in C^\infty_\diamond(SM)$.
\end{Lemma}
\begin{proof}
Apply Pestov identity and Theorem~\ref{existence of alpha-control} for $u\in C^\infty(SM)$
\begin{equation}\label{temp-surj}
\begin{split}
\|V\bfG_E u\|^2 & =\|\mathbf G_EVu\|^2-(\mathbb KVu,Vu)+\|\mathbf G_Eu\|^2\\
& \ge \|\mathbf G_Eu\|^2+\alpha(\|\mathbf G_EVu\|^2+\|Vu\|^2).
\end{split}
\end{equation}
Recall the commutation relation $[\mathbf G_E,V]u=X_\perp u-V(\lambda)Vu$, which implies that
$$
\|X_\perp u-V(\lambda)Vu\|^2\le 2(\|\mathbf G_EVu\|^2+\|V\mathbf G_Eu\|^2).
$$
Therefore,
\begin{equation}\label{temp-surj 2}
\|\mathbf G_EVu\|^2\ge\frac{1}{2}\|X_\perp u-V(\lambda)Vu\|^2-\|V\mathbf G_Eu\|^2.
\end{equation}
Thus, there are constants $C', C''>0$ such that
$$
C'\|\nabla u\|^2\leq\|\bfG_E u\|^2+\|X_\perp u-V(\lambda)Vu\|^2+\|Vu\|^2\le C''\|Pu\|^2,
$$
here $\nabla u=(Xu, X_{\perp}u, Vu)$.
By the Poincar\'e inequality, there are constants $D, D'>0$ satisfying
$$
\|u\|^2\le D(\|\mathbf G_Eu\|^2+\|X_\perp u\|^2+\|Vu\|^2)\leq D'\|\nabla u\|^2
$$
for all $u\in C^\infty_\diamond(SM)$. Hence, there is $C>0$ such that
$$
\|u\|_{H^1(SM)}\le C\|Pu\|
$$
for all $u\in C^\infty_\diamond(SM)$.
\end{proof}

Lemma \ref{P} implies a solvability result for the adjoint $P^*$.
\begin{Lemma}\label{Surjectivity of P*}
Let $(M,g,E)$ be an Anosov Gaussian thermostat on a closed oriented Riemannian surface. For any $f\in H^{-1}_\diamond(SM)$ there is $h\in L^2(SM)$ such that
$$
P^* h=f\quad\text{in}\quad SM.
$$
Moreover, $\|h\|\le C \|f\|_{H^{-1}(SM)}$ with $C>0$ being independent of $f$.
\end{Lemma}
\begin{proof}
Consider the subspace $PC^\infty_\diamond(SM)$ of $L^2(SM)$. By Lemma~\ref{P}, any element $w$ of $PC^\infty_\diamond(SM)$ has the form $w=Pu$ for some $u\in C^\infty_\diamond(SM)$. For a given $f\in H^{-1}_\diamond(SM)$, consider the linear functional
$$
L:PC^\infty_\diamond(SM)\to\C,\quad L(Pu)=\<u,f\>.
$$
Lemma~\ref{P} implies that the functional $L$ satisfies
$$
|L(Pu)|\le \|f\|_{H^{-1}(SM)}\,\|u\|_{H^1(SM)}\le C\|f\|_{H^{-1}(SM)}\, \|Pu\|.
$$
This says that $L$ is continuous on $PC^\infty_\diamond(SM)$. Therefore, by Hahn-Banach Theorem, the operator $L$ has a continuous extension
$$
\mathcal L:L^2(SM)\to \C,\quad |\mathcal L(v)|\le C\|f\|_{H^{-1}(SM)}\,\|v\|.
$$
Now, we apply the Riesz Representation Theorem to find $h\in L^2(SM)$ satisfying
$$
\mathcal L(v)=(v,h),\quad \|h\|\le C\|f\|_{H^{-1}(SM)}.
$$
If $u\in C^\infty_\diamond(SM)$, we have
$$
\<u,P^*h\>=\<Pu,h\>=L(Pu)=\<u,f\>.
$$
It follows that $P^*h=f$, since $f$ is orthogonal to constants.
\end{proof}

Now, we are ready to prove the surjectivity of $I^*_0$. Actually Theorem~\ref{th4} is a particular case of the next result (let $a=0$).
\begin{Theorem}\label{surjectivity of I*_0 for generalized thermostats}
Let $(M,g,E)$ be an Anosov Gaussian thermostat on a closed oriented Riemannian surface. Given $a\in H^{-1}_\diamond(SM)$ and $f\in L^2(M)$, there exists $w\in H^{-1}(SM)$ with $(\bfG_E+V(\lambda))w = a$ and $w_0 = f$.
\end{Theorem}
\begin{proof}
For a given $f\in C^\infty(M)$, by Lemma~\ref{Surjectivity of P*}, there is $h\in L^2(SM)$ satisfying
$$
P^*h=a-(\bfG_E+V(\lambda))f\quad \text{in}\quad SM.
$$
Setting $w:=Vh+f$, we get
\begin{equation*}
\begin{split}
(\bfG_E+V(\lambda))w&=(\bfG_E+V(\lambda))Vh+(\bfG_E+V(\lambda))f\\
&=P^*h+(\bfG_E+V(\lambda))f=a
\end{split}
\end{equation*}
and it is easy to see that $w_0=f$.
\end{proof}



\section{Surjectivity of $I_1^*$}
Let $(M,g,E)$ be a Gaussian thermostat on a compact oriented surface. Consider the following first order differential operators introduced by Guillemin and Kazhdan \cite{GK}
$$
\eta_+=\frac{1}{2}(X+iX_\perp),\quad \eta_-=\frac{1}{2}(X-iX_\perp).
$$
It was shown that $\eta_\pm:\Omega_k\to \Omega_{k\pm1}$ for $k\in\Z$, and that these operators are elliptic. We introduce the following differential operators $\mu_{\pm}:\Omega_k\to \Omega_{k\pm1}$ for $k\in\Z$, corresponding to the Gaussian thermostat $(M,g,E)$, given by
\begin{equation}\label{def-mu}
\mu_+=\eta_++\lambda_1 V,\quad \mu_-=\eta_-+\lambda_{-1}V,
\end{equation}
where $\lambda=\lambda_1+\lambda_{-1}$ (notice that $\lambda$ corresponds to a $1$-form). Thus $\mu_++\mu_-=\bfG_E=X+\lambda V$.


For fixed $m\ge 1$, we define the projection operator $T_m:C^\infty(SM)\to \bigoplus_{|k|\ge m+1}\Omega_k$ by
$$
T_mu=\sum_{|k|\ge m+1}u_k.
$$
We also consider the operator $Q_m:C^\infty(SM)\to \bigoplus_{|k|\ge m+1}\Omega_k$ defined by $Q_mu:=T_mV\bfG_E u$.

The next proposition will be the key ingredient for the proofs of the main results.
\begin{Proposition}\label{prop1}
Let $(M,g,E)$ be an $\alpha$-controlled Gaussian thermostat on a closed oriented Riemannian surface, and let $m\ge1$ be an integer. Then for any given $u\in \bigoplus_{|k|\ge m}\Omega_k$ the following holds
\begin{multline*}
\|Q_mu\|^2\ge (1-(m-1)^2+\alpha m^2)(\|\mu_-u_m\|^2+\|\mu_+u_{-m}\|^2)\\
+(1-m^2+\alpha(m+1)^2)(\|\mu_-u_{m+1}\|^2+\|\mu_+u_{-m-1}\|^2)+\|v\|^2+\alpha\|w\|^2,
\end{multline*}
where $v=T_m\mathbf G_E u$ and $w=T_m\mathbf G_E Vu$.
\end{Proposition}
\begin{proof}
Let $u\in \bigoplus_{|k|\ge m}\Omega_k$. Since $\mathbf G_E=\mu_++\mu_-$, 
\begin{multline*}
\|\mathbf G_E u\|^2=\|\mu_-u_{m+1}\|^2+\|\mu_-u_m\|^2+\|\mu_+u_{-m-1}\|^2+\|\mu_+u_{-m}\|^2+\|v\|^2.
\end{multline*}
Similarly
\begin{multline*}
\|\mathbf G_E Vu\|^2=(m+1)^2\|\mu_-u_{m+1}\|^2+m^2\|\mu_-u_m\|^2+(m+1)^2\|\mu_+u_{-m-1}\|^2\\
+m^2\|\mu_+u_{-m}\|^2+\|w\|^2.
\end{multline*}
Since $V\mathbf G_E u=\sum_{|k|\le m}ik(\mathbf G_E u)_k+Q_mu$, we have
\begin{multline*}
\|V\mathbf G_E u\|^2=m^2\|\mu_-u_{m+1}\|^2+(m-1)^2\|\mu_-u_m\|^2+m^2\|\mu_+u_{-m-1}\|^2\\
+(m-1)^2\|\mu_+u_{-m}\|^2+\|Q_mu\|^2.
\end{multline*}
By the Pestov identity and the hypostheses, we get
$$
\|V\mathbf G_Eu\|^2\ge \alpha\|\mathbf G_EVu\|^2+\|\mathbf G_Eu\|^2.
$$
Making the appropriate substitutions we obtain our result.
\end{proof}

\begin{Lemma}\label{temp-surj-I*1}
Let $(M,g,E)$ be an Anosov Gaussian thermostat. Suppose that there is a constant $C>0$ such that
$$
\|\mathbf G_Eu\|\le C\|Q_mu\|
$$
for all $u\in\bigoplus_{|k|\ge m}\Omega_k$. Then there exists another constant $D>0$ such that
$$
\|u\|_{H^1(SM)}\le D\|Q_mu\|
$$
for all $u\in\bigoplus_{|k|\ge m}\Omega_k$.
\end{Lemma}
\begin{proof}
Let $u\in\bigoplus_{|k|\ge m}\Omega_k$. By the definitions of $T_m$ and $Q_m$ we have
$$
\|Pu\|^2=\sum_{|k|\leq m}k^2\|(\mathbf G_Eu)_k\|^2+\|Q_mu\|^2\le C_1\|\mathbf G_Eu\|^2+\|Q_mu\|^2
$$
for some constant $C_1>0$. The hypothesis guarantees the existence of a constant $C_2>0$ such that
$$
\|Pu\|\le C_2\|Q_mu\|
$$
for any $u\in\bigoplus_{|k|\ge m}\Omega_k$. Now we apply Lemma~\ref{P} to finish the proof.
\end{proof}

\begin{Lemma}\label{Corollary}
Let $(M,g,E)$ be an Anosov Gaussian thermostat which is $\alpha$-controlled, for some $\alpha>(m-1)/(m+1)$
then there is a constant $C>0$ such that
$$
\|u\|_{H^1(SM)}\le C\|Q_mu\|
$$
holds for any $u\in\bigoplus_{|k|\geq m}\Omega_k$.
\end{Lemma}
\begin{proof}
By Proposition \ref{prop1}, for $\alpha>(m-1)/(m+1)$, there is a constant $C>0$ satisfying  
\begin{equation}
\|Q_mu\|\geq C\|\bfG_E u\|.
\end{equation}
Now, one can conclude the proof by applying Lemma \ref{temp-surj-I*1}.
\end{proof}

\begin{Remark}
\rm As an immediate corollary of Lemma \ref{Corollary} and the smooth Livsic theorem \cite{LMM}, one obtains that on an Anosov Gaussian thermostat which is $\alpha$-controlled for $\alpha>(m-1)/(m+1)$, $I_m$ is s-injective. In particular, an Anosov Gaussian thermostat with non-positive thermostat curvature is $1$-controlled, this is enough for proving Corollary \ref{cor1}. 

However, in section 7 we will prove Theorem \ref{th1} which is a stronger version of the injectivity of $I_m$, namely $\alpha=(m-1)/(m+1)$.
\end{Remark}

\begin{Lemma}\label{Surjectivity for Q*}
Let $(M,g,E)$ be an Anosov Gaussian thermostat which is $\alpha$-controlled for some $\alpha>(m-1)/(m+1)$. For any $f\in H^{-1}(SM)$ with $f_k=0$ for $|k|\le m-1$, there is $h\in L^2(SM)$ such that
$$
Q_m^* h=f\quad\text{in}\quad SM.
$$
Moreover, $\|h\|\le C \|f\|_{H^{-1}(SM)}$ with $C>0$ being independent of $f$.
\end{Lemma}
\begin{proof}
Consider the subspace $Q_m\bigoplus_{|k|\ge m}\Omega_k$ of $L^2(SM)$. By Lemma~\ref{P}, any element $v$ of $Q_m\bigoplus_{|k|\ge m}\Omega_k$ has the form $v=Q_mu$ for some $u\in \bigoplus_{|k|\ge m}\Omega_k$. For a given $f\in H^{-1}_\diamond(SM)$, we consider the linear functional
$$
L:Q_m\bigoplus_{|k|\ge m}\Omega_k\to\C,\quad L(Pu)=\<u,f\>.
$$
Lemma~\ref{Corollary} implies that this functional satisfies
$$
|L(Q_mu)|\le \|f\|_{H^{-1}(SM)}\,\|u\|_{H^1(SM)}\le C\|f\|_{H^{-1}(SM)}\, \|Q_mu\|.
$$
This means that $L$ is continuous on $\bigoplus_{|k|\ge m}\Omega_k$. Therefore, by Hahn-Banach theorem $L$ has a continuous extension
$$
\mathcal L:L^2(SM)\to \C,\quad |\mathcal L(v)|\le C\|f\|_{H^{-1}(SM)}\,\|v\|.
$$
Now, we apply the Riesz representation theorem to find $h\in L^2(SM)$ satisfying
$$
\mathcal L(v)=(v,h),\quad \|h\|\le C\|f\|_{H^{-1}(SM)}.
$$
If $u\in C^\infty(SM)$, we have
\begin{multline*}
\<u,Q_m^*h\>=\<Q_mu,h\>=\<Q_m(u-\sum_{|k|\le m-1}u_k),h\>=L(Q_m(u-\sum_{|k|\le m-1}u_k))\\
=\<u-\sum_{|k|\le m-1}u_k,f\>=\<u,f\>.
\end{multline*}
The last equality holds because $f_k=0$ for all $k$ satisfying $|k|\le m-1$.
\end{proof}

Now, we give the proof of our main result on the surjectivity of $I_1^*$.
\begin{proof}[Proof of Theorem~\ref{th5}]
Set $a:=-(\bfG_E+V(\lambda))\alpha$. Since $\delta\alpha=0$, by \cite{PSU} this is equivalent to $\eta_+ \alpha_{-1}+\eta_- \alpha_1=0$. On the other hand, $(\lambda_1 V+V(\lambda_1))\alpha_{-1}=(\lambda_{-1}V+V(\lambda_{-1}))\alpha_1=0$, which imples $a_0=0$. By Theorem~\ref{existence of alpha-control}, an Anosov thermostat is $\alpha$-controlled for some $\alpha>0$. Therefore, we can apply Lemma~\ref{Surjectivity for Q*} with $m=1$ to find $h\in L^2(SM)$ such that
$$
Q_m^*h=(\bfG_E+V(\lambda))VTh=-(\bfG_E+V(\lambda))\alpha.
$$
Set $w:=VTh+\alpha$, then $(\bfG_E+V(\lambda))w=0$ and $w_{-1}+w_1=\alpha$.
\end{proof}



\section{Injectivity of operators $\mu_+, \mu_-$}

The following result on the injectivity of $\mu_+,\mu_-$ is one of the crucial components in the proof of Theorem~\ref{th1}. It does generalize the corresponding result obtained in \cite{GK}.
\begin{Proposition}\label{injectivity of mu operators}
Let $(M,g,E)$ be a Gaussian thermostat on a closed oriented Riemannian surface of genus $\geq 2$. Consider the operators $\mu_\pm:\Omega_k\to \Omega_{k\pm1}$ defined as in \eqref{def-mu}, then $\mu_+:\Omega_k\to \Omega_{k+1}$ is injective for $k\ge 1$ and $\mu_-:\Omega_k\to \Omega_{k-1}$ is injective for $k\le -1$.
\end{Proposition}
This is a consequence of the following lemmas. The first lemma says that the kernel of $\mu_\pm$ is invariant under the conformal change of the metric and the Gaussian thermostat: $(g,E)\mapsto (e^{2\sigma}g,e^{-2\sigma}E)$.
\begin{Lemma}\label{invariant}
Let $(M,g,E)$ be a Gaussian thermostat on an oriented surface, and let $u\in\Omega_m$ be such that $\mu_+u=0$. Then $\tilde u=e^{m\sigma}u$ satisfies $\tilde \mu_+\tilde u=0$ for any smooth function $\sigma\in C^\infty(M,\R)$. Here $\tilde \mu_+$ denotes the operator defined as in \eqref{def-mu} for the Gaussian thermostat $(M,\tilde g,\tilde E)$ with $\tilde g= e^{2\sigma}g$ and $\tilde E=e^{-2\sigma}E$.
\end{Lemma}

Before giving the proof we introduce some conventions. If $A$ is a notation for some object in the context of the thermostat $(M,g,E)$, by $\tilde A$ we denote the same object but in the context of the thermostat $(M,\tilde g,\tilde E)$. For example, since $SM$ denotes the unit sphere bundle with respect to the metric $g$, then $\tilde SM$ denotes the unit sphere bundle with respect to the metric $\tilde g$. Another example, by $\alpha$ we denote the $1$-form dual to the external vector field $E$ with respect to the metric $g$. Then $\tilde \alpha$ denotes the $1$-form dual to the external vector field $\tilde E$ with respect to the metric $\tilde g$.

\begin{proof}
Consider the isothermal coordinates $(x,y)$ on $(M,g)$ such that the metric can be written as $ds^2=e^{2\rho}(dx^2+dy^2)$ where $\rho\in C^\infty(M,\R)$. This gives coordinates $(x,y,\varphi)$ on $SM$ where $\varphi$ is the angle between a unit vector $v$ and $\de{}{x}$. In these coordinates, we have $V=\de{}{\varphi}$ and
\begin{align*}
X&=e^{-\rho}\(\cos\varphi\de{}{x}+\sin\varphi\de{}{y}+\(-\de{\rho}{x}\sin\varphi+\de{\rho}{y}\cos\varphi\)\de{}{\varphi}\),\\
X_\perp&=-e^{-\rho}\(-\sin\varphi\de{}{x}+\cos\varphi\de{}{y}-\(\de{\rho}{x}\cos\varphi+\de{\rho}{y}\sin\varphi\)\de{}{\varphi}\).
\end{align*}
Consider $u\in\Omega_m$ and write $u(x,y,\varphi)=h(x,y)e^{im\varphi}$. Then a straightforward calculation, using these formulas, shows that
\begin{equation}\label{eta+ expression}
\eta_+(u)=e^{(m-1)\rho}\p(he^{-m\rho})e^{i(m+1)\varphi},
\end{equation}
where $\p=\frac{1}{2}\(\de{}{x}-i\de{}{y}\)$.

In order to write $\mu_+$ we set $\alpha_z:=\frac{1}{2}(E^1-iE^2)$ where $E^1$ and $E^2$ are coordinates of the vector field $E$, i.e. $E=(E^1,E^2)$. A straightforward calculation shows that
$$
\alpha_{+}(x,y,\varphi)=\alpha_z(x,y)\,e^{\rho}e^{i\varphi}.
$$
Combine this with \eqref{eta+ expression} and \eqref{def-mu}, we obtain
\begin{equation}\label{mu+ expression}
\mu_+(u)=e^{(m-1)\rho}(\p-me^{2\rho}\alpha_z)(he^{-m\rho})e^{i(m+1)\varphi}.
\end{equation}

The same coordinates $(x,y)$ will be isothermal on $(M,\tilde g)$ and the metric $\tilde g$ can be written as $d\tilde s^2=e^{2\rho+2\sigma}(dx^2+dy^2)$. Then the coordinates on $\tilde SM$ will be $(x,y,\varphi)$ where $\varphi$ is as before. In these coordinates we have
\begin{equation}\label{mu+ expression tilde}
\tilde \mu_+(\tilde u)=e^{(m-1)(\rho+\sigma)}(\p-me^{2\rho+2\sigma}\tilde \alpha_z)(\tilde he^{-m(\rho+\sigma)})e^{i(m+1)\varphi}
\end{equation}
for any $\tilde u\in \tilde\Omega_m$ written as $\tilde u(x,y,\varphi)=\tilde h(x,y)e^{im\varphi}$.

Assume that $\mu_+u=0$, where $u\in \Omega_m$ is written as $u(x,y,\varphi)=h(x,y)e^{im\varphi}$. Then from \eqref{mu+ expression} we conclude that $(\p-me^{2\rho}\alpha_z)(he^{-m\rho})=0$. 

Now consider $\tilde u=e^{m\sigma}u$. Then $\tilde u=\tilde h e^{im\varphi}$ with $\tilde h=e^{m\sigma}h$, and $\tilde \alpha_z=e^{-2\sigma}\alpha_z$. Therefore, by \eqref{mu+ expression tilde}, we have
\begin{align*}
\tilde \mu_+(\tilde u)&=e^{(m-1)(\rho+\sigma)}(\p-me^{2\rho+2\sigma}\tilde \alpha_z)(\tilde he^{-m(\rho+\sigma)})e^{i(m+1)\varphi}\\
&=e^{(m-1)(\rho+\sigma)}(\p-me^{2\rho}\alpha_z)(he^{-m\rho})e^{i(m+1)\varphi}.
\end{align*}
Thus, we conclude that $\tilde\mu_+\tilde u=0$.
\end{proof}

\begin{Lemma}\label{divergence}
Let $(M,g,E)$ be a Gaussian thermostat on an oriented surface. If $(M,\tilde g,\tilde E)$ is the conformal Gaussian thermostat, that is $\tilde g= e^{2\sigma}g$ and $\tilde E=e^{-2\sigma}E$, then $\div_{\tilde g}\tilde E=e^{-2\sigma}\div_gE$.
\end{Lemma}
\begin{proof}
The proof follows by straightforward computations in isothermal coordinates $(x,y)$ on $(M,g)$. The Christoffel symbols are
$$
\Gamma_{11}^1=-\Gamma_{22}^1=\Gamma^2_{12}=\de{\rho}{x},\quad \Gamma_{22}^2=-\Gamma_{11}^2=\Gamma_{12}^1=\de{\rho}{y}.
$$
If $E=(E^1,E^2)$ in coordinates $(x,y)$, then the expression for $\div_gE$ is
$$
\div_g E=\nabla_{1}E^1+\nabla_2 E^2=\de{E^1}{x}+\de{E^2}{y}+2\(\de{\rho}{x}E^1+\de{\rho}{y}E^2\).
$$
Note that the metric $\tilde g$ can be written as $d\tilde s^2=e^{2\rho+2\sigma}(dx^2+dy^2)$. Therefore the Christoffel symbols for $\tilde g$ are
$$
\tilde \Gamma_{11}^1=-\tilde \Gamma_{22}^1=\tilde \Gamma^2_{12}=\de{(\rho+\sigma)}{x},\quad \tilde \Gamma_{22}^2=-\tilde \Gamma_{11}^2=\tilde \Gamma_{12}^1=\de{(\rho+\sigma)}.{y},
$$
Since $\tilde E=(e^{-2\sigma}E^1,e^{-2\sigma}E^2)$ in coordinates $(x,y)$, the expression for $\div_{\tilde g}\tilde E$ is
\begin{equation*}
\begin{split}
\div_{\tilde g}\tilde E&=\tilde \nabla_{1}E^1+\tilde \nabla_2 E^2=e^{-2\sigma}\(\de{E^1}{x}+\de{E^2}{y}+2\(\de{\rho}{x}E^1+\de{\rho}{y}E^2\)\)\\
& =e^{-2\sigma}\div_gE.
\end{split}
\end{equation*}
\end{proof}

\begin{Lemma}\label{negative curvature}
Let $(M,g,E)$ be a Gaussian thermostat on a closed oriented surface of genus $\geq 2$, then there exists a function $\sigma\in C^{\infty}(M,\mathbb{R})$, such that the conformal Gaussian thermostat $(M,e^{2\sigma}g,e^{-2\sigma}E)$ has negative thermostat curvature.
\end{Lemma}
\begin{proof}
Let $K, \tilde{K}$ be the Gaussian curvatures of $(M,g)$ and $(M,e^{2\sigma}g)$ respectively. It is well known that $\tilde{K}=e^{-2\sigma}(K-\Delta_g\sigma)$, here $\Delta_g$ is the Laplacian under the metric $g$. On the other hand, a straightforward calculation shows that the thermostat curvature of $(M,g,E)$ has the form
$$\mathbb{K}=K-\div_gE.$$ Above discussion together with Lemma \ref{divergence} implies that the thermostat curvature of $(M,e^{2\sigma}g,e^{-2\sigma}E)$ is
$$\tilde{\mathbb{K}}=\tilde{K}-\div_{\tilde{g}}\tilde{E}=e^{-2\sigma}(K-\Delta_g\sigma-\div_gE).$$

To prove the lemma, we need to find a real-valued smooth function $\sigma$ and a constant $c<0$ for the following equation 
\begin{equation}\label{genus}
K-\Delta_g\sigma-\div_gE=c<0.
\end{equation}
Notice that on a closed connected Riemannian surface, the solvability condition for \eqref{genus} is
$$0=\int_M K-c-\div_gE\,d{\rm Vol}_g=\int_M K-c\,d{\rm Vol}_g.$$
By the Gauss-Bonnet theorem and the assumption that the genus $\geq 2$ (i.e. the Euler characteristic $\chi(M)<0$), we can choose
$$c=\frac{\int_M K\,d{\rm Vol}_g}{{\rm Vol}_g(M)}=\frac{2\pi\chi(M)}{{\rm Vol}_g(M)}<0,$$
where ${\rm Vol}_g(M)$ is the volume of $M$ under the metric $g$.

Thus there exists $\sigma\in C^{\infty}(M,\mathbb{R})$ such that 
$$\tilde{\mathbb{K}}=e^{-2\sigma}\frac{2\pi\chi(M)}{{\rm Vol}_g(M)}<0.$$
\end{proof}

Lemma \ref{invariant} and \ref{negative curvature} imply that to prove Proposition \ref{injectivity of mu operators}, we only need to show that it's true for the case $\mathbb{K}<0$.

\begin{Lemma}
Given a Gaussian thermostat $(M,g,E)$ on a closed oriented surface with $\mathbb{K}=K-\div_g E<0$, where $K$ is the Gaussian curvature of $(M,g)$, then $\mu_+:\Omega_k\to \Omega_{k+1}$ is injective for $k\ge 1$ and $\mu_-:\Omega_k\to \Omega_{k-1}$ is injective for $k\le -1$.
\end{Lemma}
\begin{proof}
Let $u\in\Omega_k$, since $\bfG_E=\mu_++\mu_-$, the following expressions hold
\begin{align*}
\bfG_E u=\mu_+u+&\mu_-u,\quad \bfG_E Vu=ik\mu_+ u+ik\mu_- u,\\
V\bfG_E u&=i(k+1)\mu_+ u+i(k-1)\mu_- u.
\end{align*}
Substituting these into the Pestov identity, we obtain an integral identity
$$
2k\|\mu_- u\|^2=2k\|\mu_+ u\|^2+k^2(\mathbb{K} u,u).
$$
According to our hypothesis $\mathbb{K}<0$, we come to the following inequality
\begin{equation}\label{temp-ineq}
2k\|\mu_- u\|^2\le2k\|\mu_+ u\|^2.
\end{equation}
Consider the case $k\ge 1$ and assume $\mu_+ u=0$, we get
$$
0\le\|\mu_- u\|^2\le 0,
$$
Thus $u\equiv0$ as desired. Using similar ideas for the case $k\le -1$ one can prove that $\mu_- u=0$ implies $u\equiv 0$.
\end{proof}


\section{Injectivity of $I_m$}
Before giving the proof of the $s$-injectivity of $I_m$, it is worth pointing out that if the terminator value of a Gaussian thermostat $(M,g,E)$ is $\beta_{\rm ter}$, then $(M,g,E)$ is free of $\beta_{\rm ter}$-conjugate points. Indeed assume that $(M,g,E)$ has $\beta_{\rm ter}$-conjugate points, i.e. there exists a thermostat geodesic $\gamma$ and a non-trivial solution $y(t)$ to the $\beta_{\rm ter}$-Jacobi equation along $\gamma$ such that $y(0)=y(T)=0$ for some $T>0$. Notice that $\dot{y}(T)\neq 0$, thus there is a small neighborhood $U$ of $\beta_{\rm ter}$, such that for all $\beta\in U$ there are $\beta$-conjugate points. This contradicts the definition of the terminator values. 

Since $(M,g,E)$ has no $\beta_{\rm ter}$-conjugate points, by Remark \ref{no b conj impies (b-1)/b-controlled}, it is $(\beta_{\rm ter}-1)/\beta_{\rm ter}$-controlled. Notice that for Anosov Gaussian thermostats, there are no conjugate points in the usual sense, which means that $\beta_{\rm ter}\geq 1$ (actually one can get $\beta_{\rm ter}>1$ for Anosov Gaussian thermostats).


The following injectivity result will imply Theorem \ref{th1}.
\begin{Theorem}\label{alpha controlled injectivity}
Let $(M,g,E)$ be a Gaussian thermostat on a closed surface of genus $\mathtt g\ge 2$ which is $(m-1)/(m+1)$-controlled. Let $\varphi$ be a symmetric $m$-tensor and suppose that there is a smooth solution $h$ to the transport equation
$$
\bfG_{E} h=\varphi.
$$
Then $h$ is of degree $m-1$.
\end{Theorem}
\begin{proof}
Let $u=\sum_{|k|\geq m}h_k$, then $\bfG_Eu$ has degree $m$ and $Q_mu=0$. By Proposition~\ref{prop1} and the assumption $\alpha=(m-1)/(m+1)$, we get that
$$
\mu_-u_m=0\quad\text{and}\quad \mu_+u_{-m}=0.
$$
Thus
$$
\bfG_E u=\mu_-u_{m+1}+\mu_+u_{-(m+1)}
$$
and
$$
\bfG_E Vu=i(m+1)\mu_-u_{m+1}-i(m+1)\mu_+u_{-(m+1)}.
$$
Therefore,
\begin{multline*}
X_\perp u-V(\lambda)Vu=[\bfG_E,V]u\\
=i(m+1)\mu_-u_{m+1}-i(m+1)\mu_+u_{-(m+1)}-im\mu_-u_{m+1}+im\mu_+u_{-(m+1)}\\
=i\mu_-u_{m+1}-i\mu_+u_{-(m+1)}.
\end{multline*}
It is known that $X_\perp u-V(\lambda)Vu=i\mu_- u-i\mu_+ u$. Hence $\mu_-u=\mu_-u_{m+1}$ and $\mu_+u=\mu_+u_{-(m+1)}$, in particular, $\mu_+u_k=0$ and $\mu_-u_{-k}=0$ for $k\geq m$. Then Proposition~\ref{injectivity of mu operators} implies that $u\equiv0$, thus $h$ is of degree $m-1$.
\end{proof}

\begin{proof}[Proof of Theorem~\ref{th1}]
Let $\varphi$ be a symmetric $m$-tensor, such that $I_m\varphi\equiv 0$. By the smooth Livsic theorem, there is $h\in C^\infty(SM)$ such that $\bfG_Eh=\varphi$. 

On the other hand, a closed oriented surface whose unit sphere bundle carries an Anosov flow must have genus $\geq 2$. Indeed, by a classic result of Plante and Thurston \cite{PT}, if an $S^1$-bundle over a closed oriented surface carries an Anosov flow, the fundamental group of the bundle must grow exponentially. However the fundamental group of any $S^1$-bundle over a $2$-sphere or torus only has polynomial growth. 

Finally, by Remark \ref{no b conj impies (b-1)/b-controlled} and the discussion about terminator values at the beginning of this section, $(M,g,E)$ is $(m-1)/(m+1)$-controlled. 

Now Theorem \ref{th1} is a direct consequence of Theorem \ref{alpha controlled injectivity}.
\end{proof}





\section{Results for surfaces with boundary}
As mentioned in the introduction, some of the arguments above also work for compact surfaces with boundary. The main change when dealing with the boundary case is that the functions need to vanish on the boundary whenever appropriate. 

In this section we assume that $(M,g)$ is a compact oriented Riemannian surface with smooth boundary $\p M$, we will prove Theorem \ref{th2} which is an injectivity result for Gaussian thermostats $(M,g,E)$ on surfaces with boundary. Let $\Lambda$ denote the second fundamental form of $\p M$ and $\nu(x)$ the inward unit normal to $\p M$ at $x$. We say that $\p M$ is {\it strictly thermostat convex} if
\begin{equation}
\Lambda(x,v)>\langle E(x)-\<E(x),v\>v,\nu(x)\rangle
\end{equation}
for all $(x,v)\in S(\p M)$, here $E$ is the external field.

For $x\in M$, we define the {\it thermostat exponential map} by
$$
\exp^E_x(tv)=\pi\circ\phi_t(v),\quad t\ge 0,\mbox{ }v\in S_x M
$$
which is $C^1$-smooth on $T_x M$ and $C^{\infty}$-smooth on $T_x M\setminus \{0\}$.

We say that $(M,g,E)$ is {\it simple} if 1) $\p M$ is strictly thermostat convex and 2) the thermostat exponential map $\exp^E_x:(\exp^E_x)^{-1}(M)\to M$ is a diffeomorphism for every $x\in M$. These two conditions guarantee that every two points on $M$ are connected by a unique thermostat geodesic and there is no conjugate points. In this case, $M$ is diffeomorphic to the unit ball of $\mathbb R^n$, which is simply connected.

Results in Section 2 are still valid in the boundary case if the trace of $u$ or $v$ vanishes. The Pestov identity also holds:

\begin{Theorem}\label{pestov boundary}
Let $(M,g,E)$ be a Gaussian thermostat on a compact oriented surface with boundary. If $u\in C^\infty(SM,\C)$ and $u|_{\p SM}=0$, then
$$
\|\mathbf G_E Vu\|^2-(\K Vu,Vu)=\|V\mathbf G_E u\|^2-\|\mathbf G_E u\|^2.
$$
\end{Theorem}

Notice that the estimate of Theorem \ref{existence of alpha-control} plays an important role in the arguments for the case of closed surfaces. To establish our result for the boundary case, we need a similar estimate. Given a Riemannian surface $M$ with boundary, denoting $\p SM$ the boundary of $SM$, we define a subset of $\p SM$,
$$\p_+SM:=\{(x,v)\in \p SM: \<v,\nu(x)\>_g\geq 0\}.$$
Note that $\nu(x)$ is the inward unit normal to $\p M$ at $x$. We start with the following existence result of distinct solutions to the Riccati equation on 2D simple Gaussian thermostats. 


\begin{Lemma}\label{distinct solutions}
Let $(M,g,E)$ be a simple Gaussian thermostat on a compact oriented surface with boundary. Then there exist smooth nowhere equal solutions $r^+$ and $r^-$ to the Riccati type equation
\begin{equation}\label{Riccati}
\bfG_E r+r^2-V(\lambda)r+\mathbb K-\bfG_E V(\lambda)=0.
\end{equation} 
\end{Lemma}
\begin{proof}
We embed $M$ into larger compact surfaces $\tilde M$, $\widetilde M$ with boundary such that $M\subset\tilde M^{int}\subset \tilde M\subset \widetilde M^{int}\subset \widetilde M$, and extend $g$ and $E$ smoothly onto $\widetilde M$ such that $(\tilde M,g,E)$ and $(\widetilde M,g,E)$ are simple too.

We consider a maximum thermostat geodesic $\gamma_{z}:[0,\,l]\to \widetilde M$ with $z=(\gamma_z(0),\dot{\gamma}_z(0))\in \p_+S\widetilde M$. Let $y_z$ be the solution to the thermostat Jacobi equation 
$$\ddot{y}_z-V(\lambda)\dot{y}_z+(\mathbb K-\bfG_E V(\lambda))y_z=0$$
along $\gamma_z$ satisfying $y_z(0)=0$, $\dot{y}_z(0)=1$. By the simplicity of $(\widetilde M,g,E)$, $\gamma_z$ has no conjugate points, thus $r(z,t)=\frac{\dot y_z(t)}{y_z(t)}$ is a solution to the Riccati equation on $(0,\,l]$ with $\lim_{t\to 0}r(t)=+\infty$. Notice that $r(z,t)$ smoothly depends on $z\in \p_+S\widetilde M$. We do the same thing for all the thermostat geodesics on $\widetilde M$, which can be parametrized by $z\in \p_+S\widetilde M$, to get a well-defined smooth solution $r^+(x,\xi)=r(z(x,\xi),\tau^-(x,\xi))$ to the Riccati equation \eqref{Riccati} on $S\widetilde M^{int}$, where $(x,\xi)=(\gamma_z(\tau^-(x,\xi)),\dot{\gamma}_z(\tau^-(x,\xi)))$, $\tau^-(x,\xi)$ is the length of the unique thermostat geodesic segment connecting $\pi(z)$ and $x$ with $\xi\in S_x\widetilde M$ tangent to $\gamma_z$ at $x$. It is not difficult to see that $z$ and $\tau^-$ smoothly depend on $(x,\xi)\in S\widetilde M^{int}$. Moreover $\lim_{(x,\xi)\to \p_+S\widetilde M}r^+(x,\xi)=+\infty$.

Notice that by our definition of $\tilde M$ and $\widetilde M$, the restriction to $\tilde M$ of a thermostat geodesic $\gamma$ of $(\widetilde M, g, E)$ (if nonempty), $\gamma|_{\tilde M}$, is a thermostat geodesic of $(\tilde M,g,E)$. By a similar approach as above with the initial condition $y_z(0)=0$, $\dot{y}_z(0)=1$ at $z\in \p_+S\tilde M$ for the thermostat Jacobi equation, one can get a smooth solution $r^-$ to the Riccati equation \eqref{Riccati} on $S\tilde M^{int}$ with $\lim_{(x,\xi)\to \p_+S\tilde M}r^-(x,\xi)=+\infty$. 

Since $S\tilde M\subset S\widetilde M$ and $\p_+S\tilde M$ is compact, there exists $K>0$ such that $\sup_{(x,\xi)\in \p_+S\tilde M}r^+(x,\xi)\leq K$. We can find a smaller compact surface $U$, whose boundary $\p U$ is uniformly, sufficiently close to $\p\tilde M$, with $M\subset U\subset \tilde M^{int}$ and $(U,g,E)$ is still simple. Then there exists $c>0$ such that $\sup_{\p_+SU}r^+< K+c$ and $\inf_{\p_+SU}r^-> K+c$, i.e. $r^+$ and $r^-$ never coincide on $\p_+SU$.

Now we claim that $r^+\neq r^-$ on $SM$ (Actually $r^+\neq r^-$ on $S\tilde M^{int}$). We prove by contradictions, assume that there exists $(x,\xi)\in SM$ such that $r^+(x,\xi)=r^-(x,\xi)$. Consider the restrictions of $r^+$ and $r^-$ onto the thermostat geodesic $\gamma_{x,\xi}:[-l^-,\,l^+]\to U$, $l^-,\,l^+>0$, with $(\gamma_{x,\xi}(0),\dot{\gamma}_{x,\xi}(0))=(x,\xi)$ and $\gamma_{x,\xi}(-l^-),\,\gamma_{x,\xi}(l^+)\in\p U$. Notice that the zeroth order term of the Riccati equation \eqref{Riccati} is a polynomial with respect to $r$. Moreover, $[-l^-,\,l^+]$ is compact, thus the zeroth order term of \eqref{Riccati} is Lipschitz continuous in $r$ when it is restricted on $\gamma_{x,\xi}$. By the Picard-Lindel\"of theorem of first order ODEs, one has the global existence and uniqueness of the solution to the Riccati equation on $\gamma_{x,\xi}$ with $r(0)=r^+(x,\xi)=r^-(x,\xi)$. This implies that $r^+\equiv r^-$ along $\gamma_{x,\xi}$. In particular, there is $z\in \p_+SU$ such that $r^+(z)=r^-(z)$. However, since $r^+$ and $r^-$ are never equal on $\p_+SU$, we reach a contradiction. Therefore, $r^+$ and $r^-$ are two distinct solutions to the Riccati equation \eqref{Riccati} on $SM$.
\end{proof}

The following is an analogue of Theorem \ref{existence of alpha-control} on compact surfaces with boundary.
 
\begin{Theorem}\label{alpha controll boundary}
Let $(M,g,E)$ be a simple Gaussian thermostat on a compact oriented surface with boundary. Then there exists an $\alpha>0$ such that 
$$
\|\mathbf G_E\varphi\|^2-(\K\varphi,\varphi)\ge\alpha\left(\|\mathbf G_E\varphi\|^2+\|\varphi\|^2\right)
$$
for all $\varphi\in C^\infty(SM,\mathbb{C})$ with $\varphi|_{\p SM}=0$.
\end{Theorem}
\begin{proof}
Applying Lemma \ref{distinct solutions}, the proof is almost identical to the proof of Theorem \ref{existence of alpha-control}.
\end{proof}

Applying Theorem \ref{pestov boundary} and \ref{alpha controll boundary}, the results of Section 4 and 5 also hold for the boundary case. 
To prove Theorem \ref{th2}, we need the following lemma on the injectivity of $\mu_{\pm}$ which is an analogue of Proposition \ref{injectivity of mu operators}.

\begin{Lemma}
Let $(M,g,E)$ be a Gaussian thermostat on a compact oriented Riemannian surface with boundary. Consider the operators $\mu_\pm:\Omega_k\to \Omega_{k\pm1}$ defined as in \eqref{def-mu}. Let $k\geq 1$, if $\mu_+u=0$ where $u\in\Omega_{k},\,u|_{\p SM}=0$, then $u=0$; if $\mu_-u=0$ where $u\in\Omega_{-k},\,u|_{\p SM}=0$, then $u=0$.
\end{Lemma}
\begin{proof}
Notice that $M$ can be embedded into a closed surface of genus $\geq 2$. By Lemma \ref{negative curvature}, we only need to show the injectivity of $\mu_{\pm}$ for Gaussian thermostats of negative thermostat curvature, which is straightforward by applying Theorem \ref{pestov boundary}.
\end{proof}

With the help of above Lemma, we obtain the following injectivity result whose proof is similar to that for Theorem \ref{alpha controlled injectivity}.

\begin{Proposition}\label{injectivity for K<0}
Let $(M,g,E)$ be a Gaussian thermostat on a compact oriented surface with boundary which is $(m-1)/(m+1)$-controlled. Let $\varphi$ be a symmetric $m$-tensor and suppose that there is a smooth solution $h$, $h|_{\p SM}=0$, to the transport equation
$$
\bfG_{E} h=\varphi.
$$
Then $h$ is of degree $m-1$.
\end{Proposition}

To prove Theorem \ref{th2}, we need a version of Livsic Theorem for surfaces with boundary. Given a 2D simple Gaussian thermostat $(M,g,E)$, 
let $\tau(x,v),\, (x,v)\in SM$ be the time that the thermostat geodesic $\gamma_{x,v}$ starting at $x$ in direction $v$ exits $M$. The simplicity assumption implies that $\tau$ is finite for all $(x,v)\in SM$ and it is smooth on $SM$ except $S(\p M)$, the unit sphere bundle of the boundary $\p M$.

Given $f$ a smooth function on $SM$, it is easy to see that 
\begin{equation}\label{solution of transport equ}
u^f(x,v)=-\int_0^{\tau(x,v)}f(\gamma_{x,v}(t),\dot{\gamma}_{x,v}(t))\,dt
\end{equation}
solves the transport equation
$$\bfG_E u=f$$
in $SM$. Moreover, if $I f\equiv 0$, we obtain $u^f|_{\p SM}=0$. The ingredient is the following regularity statement.

\begin{Proposition}\label{regularity}
Let $(M,g,E)$ be a simple Gaussian thermostat on a compact oriented surface with boundary. Given $f\in C^{\infty}(SM)$ with $I f\equiv 0$, let $u^f$ be the function defined by \eqref{solution of transport equ}, then $u^f\in C^{\infty}(SM)$ too.
\end{Proposition}

The proof of Proposition \ref{regularity} for simple surfaces can be found in \cite{PSU1}, a similar argument works for simple Gaussian thermostats, thus we leave it to the reader. Now Theorem \ref{th2} follows from Proposition \ref{regularity} and \ref{injectivity for K<0}.

\section*{Acknowledgements}
The authors want to thank their advisor, Professor Gunther Uhlmann, for all his support and encouragement. Thanks are also due to Professor Gabriel Paternain and Professor Mikko Salo for helpful discussions and sharing their knowledges on the topics. The authors are also grateful to the referees for helpful comments and suggestions. The work was partially supported by NSF.


\end{document}